\documentclass[12pt]{amsart}
\usepackage{amsthm}
\usepackage{amsmath,amssymb,euscript,mathrsfs,bm}

\usepackage{pstricks}

\usepackage{eepic,bbm}

\usepackage{bbold}
\usepackage{dsfont}

\usepackage{ebgaramond}

\usepackage{setspace}
\usepackage[shortlabels]{enumitem}

\usepackage{tikz-cd}

\usepackage[colorlinks=true,linkcolor=black,citecolor=black,urlcolor=black]{hyperref}

\pushQED{\qed}

\psset{unit=1pt}
\psset{arrowsize=4pt 1}
\psset{linewidth=.5pt}

\newcommand{\define}{\textbf}

\newcommand{\C}{\mathds{C}}

\newcommand{\Q}{\mathds{Q}}

\newcommand{\Z}{\mathds{Z}}
\renewcommand{\P}{\mathds{P}}

\renewcommand{\O}{\mathcal{O}}

\newcommand{\bq}{q}
\newcommand{\bQ}{Q}

\newcommand{\pt}{\mathrm{pt}}

\DeclareMathOperator{\sh}{sh}

\DeclareMathOperator{\Inj}{Inj}
\newcommand{\Sgp}{\mathcal{S}}        

\newcommand{\LambdaXi}{\Lambda^{(\xi)}}  

\newcommand{\TZ}{{T_\Z}}     
\newcommand{\Tn}{{T}}        

\newcommand{\tto}{\twoheadrightarrow}

\newcommand{\bull}{ {\footnotesize{\ensuremath{\bullet}}}  }

\newcommand{\Fl}{ {Fl} }

\newcommand{\sGr}{\mathrm{Gr}}  
\newcommand{\sFl}{\mathrm{Fl}}

\newcommand{\aGr}{\tilde{\sGr}}  
\newcommand{\aFl}{\tilde{\sFl}}

\newcommand{\tS}{\mathds{S}}      

\newcommand{\ee}{\mathrm{e}}    

\newcommand{\isom}{\cong}

\renewcommand{\phi}{\varphi}

\renewcommand{\tilde}{\widetilde}
\renewcommand{\hat}{\widehat}
\renewcommand{\bar}{\overline}

\newtheoremstyle{scthm}%
{}{}{\itshape}{}{\scshape}{.}{ }{}

\newtheoremstyle{scdef}%
{}{}{}{}{\scshape}{.}{ }{}

\theoremstyle{scthm}

\newtheorem{theorem}{Theorem}[section]
\newtheorem{lemma}[theorem]{Lemma}
\newtheorem{proposition}[theorem]{Proposition}
\newtheorem{corollary}[theorem]{Corollary}

\newtheorem*{thm*}{Theorem}
\newtheorem*{cor*}{Corollary}
\newtheorem*{prop*}{Proposition}

\newtheorem*{corA}{Corollary~A}
\newtheorem*{corB}{Corollary~B}

\theoremstyle{scdef}
\newtheorem{definition}[theorem]{Definition}
\newtheorem{remark}[theorem]{Remark}

\begin{document}

\title[Equivariant cohomology of affine Grassmannians]{Integral equivariant cohomology of affine Grassmannians}
\author{David Anderson}
\thanks{Partially supported by NSF CAREER DMS-1945212.}

\date{May 30, 2023}
\address{Department of Mathematics, The Ohio State University, Columbus, OH 43210}
\email{anderson.2804@math.osu.edu}

\maketitle

\renewcommand{\bfseries}{\itshape}

\begin{abstract}
We give explicit presentations of the integral equivariant cohomology of the affine Grassmannians and flag varieties in type A, arising from their natural embeddings in the corresponding infinite (Sato) Grassmannian and flag variety.  These presentations are compared with results obtained by Lam and Shimozono, for rational equivariant cohomology of the affine Grassmannian, and by Larson, for the integral cohomology of the moduli stack of vector bundles on $\P^1$.
\end{abstract}

\section{Introduction}\label{s.intro}

The main aim of this note is to provide a simple presentation, in terms of generators and relations, of the torus-equivariant cohomology of the affine Grassmannian and flag variety, $\aGr_n$ and $\aFl_n$.  In particular, we obtain these rings as quotients of polynomial rings, with the quotient map arising geometrically as the pullback via embeddings in the Sato Grassmannian and flag variety, respectively.

Let $\Lambda = \Z[c_1,c_2,\ldots]$ be the polynomial ring in countably many generators, with $c_i$ in degree $2i$.  Let $p_k=p_k(c)$ be the polynomial
\begin{equation}\label{e.Newton1}
  p_k(c) = (-1)^{k-1} \det \left(\begin{array}{ccccc}
  c_1 & 1 & 0 & 0 & 0 \\
  2c_2 & c_1 & 1 & 0 & 0 \\
  3c_3 & c_2 & \ddots & \ddots & 0 \\
  \vdots & \vdots & \ddots & \ddots & 1 \\kc_k & c_{k-1} & \cdots & c_2 & c_1
  \end{array}\right).
\end{equation}
One can identify $\Lambda$ with the ring of symmetric functions in some other set of variables, making $c_k$ the complete homogeneous symmetric function, so that $p_k$ becomes the power sum symmetric function via the Newton relations.  But for now we remain agnostic about the choice of such an identification.

Fixing $n$, consider the polynomials
\begin{align}
 p_k(c|y) &= p_k(c) + p_{k-1}(c)\, e_1(y_1,\ldots,y_n) + \cdots \\ &\qquad + p_2(c)\, e_{k-2}(y_1,\ldots,y_n) + p_1(c)\, e_{k-1}(y_1,\ldots,y_n) \nonumber \\
    &= \sum_{i=1}^k p_{i}(c)\, e_{k-i}(y_1,\ldots,y_n) \label{e.pk-def}
\end{align}
in $\Lambda[y_1,\ldots,y_n]$.

Let $V$ be a complex vector space with basis $\ee_i$, for $i\in\Z$, and let $V_{\leq0}$ be the subspace spanned by $\ee_i$ for $i\leq 0$.  The torus $\Tn=(\C^*)^n$ acts by scaling the basis vector $\ee_i$ by the character $y_{{i\pmod n}}$, using representatives $1,\ldots,n$ for residues mod $n$.  Let $\sGr^d=\sGr^d(V)$ be the corresponding Sato Grassmannian parameterizing subspaces of index $d$, with the induced action of $\Tn$.  The $d$-th component of the affine Grassmannian embeds $\Tn$-equivariantly in $\sGr^d$.  (Definitions of these spaces are reviewed in \S\ref{s.bg} below.)  
We write $\tS_d \subset V$ for the tautological bundle on $\sGr^d$, and recycle the same notation for the tautological bundle on subvarieties, when the context is clear.

The equivariant cohomology of the Sato Grassmannian is $H_{\Tn}^*\sGr^d = \Lambda[y_1,\ldots,y_n]$, identifying $c_k$ with the Chern class $c_k^{\Tn}(V_{\leq 0} -\tS_d)$.  

\begin{thm*}
The inclusion $\aGr^d_n \hookrightarrow \sGr^d$ induces a surjection $H_{\Tn}^*\sGr^d \tto H_{\Tn}^*\aGr^d_n$, whose kernel is generated by $p_k(c|y)$ for $k> n$, together with $p_n(c|y)+d e_n(y)$.

In particular, the map $c_k \mapsto c_k^{\Tn}(V_{\leq 0}-\tS_d)$ defines an isomorphism of $H_{\Tn}^*(\pt) = \Z[y_1,\ldots,y_n]$-algebras
\[
  \Lambda[y_1,\ldots,y_n]/I^d_n \xrightarrow{\sim} H_{\Tn}^*(\aGr^d_n)  ,
\]
where $I_n^d$ is the ideal generated by $p_k(c|y)$ for $k> n$ and $p_n(c|y)+d e_n(y)$.
\end{thm*}

All the generators of $I_n^d$ are symmetric in the $y$ variables.  It follows that the $GL_n$-equivariant cohomology has essentially the same presentation.  Write $H_{GL_n}^*(\pt) = \Z[e_1,\ldots,e_n]$, with $e_k$ in degree $2k$, regarded as a subring of $H_{\Tn}^*(\pt)$ by sending $e_k$ to the elementary symmetric polynomial $e_k(y)$.  Define elements $p_k(c|e) \in \Lambda[e_1,\ldots,e_n]$ by the same formula \eqref{e.pk-def}, with $e_k=0$ for $k>n$.

\begin{corA}
Let $J_n^d \subset \Lambda[e_1,\ldots,e_n]$ be the ideal generated by $p_k(c|e)$ for $k>n$ and $p_n(c|e)+ de_n$.  Then the map $c_k\mapsto c_k^{GL_n}(V_{\leq 0} - \tS_d)$ defines an isomorphism of $H_{GL_n}^*(\pt)$-algebras
\[
  \Lambda[e_1,\ldots,e_n]/J^d_n \xrightarrow{\sim} H_{GL_n}^*(\aGr^d_n)  .
\]
\end{corA}

This follows from the theorem by an application of the general fact that $H_{GL_n}^*X \subset H_{\Tn}^*X$ is the invariant ring for the natural $\Sgp_n$ action on $y$ variables; see, e.g., \cite[\S15.6]{ecag}.

A presentation for the equivariant cohomology of $\aFl_n$ also follows from the theorem.  Let $\tS_\bull: \cdots \subset \tS_{-1}\subset \tS_0 \subset \tS_1 \subset \cdots$ be the tautological flag on $\aFl_n$.

\begin{corB}
Evaluating $c_k \mapsto c_k^{\Tn}(V_{\leq0}-\tS_0)$ and $x_i\mapsto c_1^{\Tn}(\tS_i/\tS_{i-1})$, we have
\[
  H_{\Tn}^*\aFl_n = \Lambda[x_1,\ldots,x_n,y_1,\ldots,y_n]/I^{\sFl}_n,
\]
where $I^{\sFl}_n$ is generated by $p_k(c|y)$ for $k\geq n$ along with $e_i(x)-e_i(y)$ for $i=1,\ldots,n$.

For $GL_n$-equivariant cohomology, the presentation is similar:
\[
  H_{GL_n}^*\aFl_n = \Lambda[x_1,\ldots,x_n,e_1,\ldots,e_n]/J^{\sFl}_n,
\]
where $J^{\sFl}_n$ is generated by $p_k(c|e)$ for $k\geq n$ along with $e_i(x)-e_i$ for $i=1,\ldots,n$.
\end{corB}

This can be deduced from the theorem by examining the action of the shift morphism on cohomology; see \S\ref{s.bg}.

\medskip
A presentation for the non-equivariant cohomology ring $H^*\aGr^0_n$ was given by Bott \cite{bott}, who used a natural coproduct structure to identify this ring with the infinite symmetric power of the cohomology of projective space.  Since $H^*\P^{n-1} \isom \Z[\xi]/(\xi^n)$, this is easily seen to be equivalent to the result of setting the $y$ variables to $0$ in the statement of the main theorem above.  (One makes the indicated identifications with symmetric functions in variables $\xi_1,\xi_2,\ldots$, and Bott's relations become $p_k(\xi)=0$ for $k\geq n$.)

Several authors have given different presentations of the equivariant cohomology ring, sometimes with field coefficients, using localization or representation theory \cite{yun,yun-zhu,lam-shimozono}.  In the context of the moduli stack of vector bundles on $\P^1$, Larson described the integral cohomology ring as a subring of a polynomial ring with rational coefficients \cite{larson}.  In fact, Larson's description is equivalent to the quotient ring appearing in Corollary A; the precise translation is given in \S\ref{s.VBs} below.

In this note, the main contributions are to provide a concise presentation of $H_{\Tn}^*\aGr_n^d$ as a quotient of a polynomial ring, and to show how Bott's method extends naturally to the equivariant setting.  We also describe a new basis of double monomial symmetric functions which are well-adapted to the presentation of $H_{\Tn}^*\aGr_n$.  Apart from some elementary calculations with symmetric functions, the only additional input required is a well-known presentation of the equivariant cohomology of projective space.

\medskip
\noindent
{\it Acknowledgements.}  Thomas Lam very helpfully pointed me to references for other presentations of the equivariant cohomology of the affine Grassmannian.  I thank Linda Chen, Hannah Larson, and Isabel Vogt for many conversations about the cohomology of the affine Grassmannian and of the moduli stack of vector bundles.

\section{Infinite and affine flag varieties}\label{s.bg}

We follow \cite{infschub}, which in turn is based on \cite{LLS} and \cite{pressley-segal}.  As in the introduction, $V$ is a complex vector space with basis $\ee_i$ for $i\in \Z$.  For any interval $[a,b]$ in $\Z$, we write $V_{[a,b]}$ for the subspace spanned by $\ee_i$ for $i$ in $[a,b]$.  We will especially use subspaces $V_{\leq p}$ (or $V_{>q}$), spanned by $\ee_i$ for $i\leq p$ (or $i>q$, respectively).

\subsection{Definitions}
The \define{Sato Grassmannian} $\sGr^d$ is the set of subspaces $E\subset V$ of index $d$.  This means (1) $V_{\leq -m} \subset E \subset V_{\leq m}$ for some (and hence all)  $m\gg0$, and (2) $\dim E/(V_{\leq 0} \cap E) - \dim V_{\leq0}/(V_{\leq0}\cap E) = d$.  The Sato Grassmannian is naturally topologized as an ind-variety.

The \define{Sato flag variety} is the subvariety $\sFl\subset \prod_{d\in\Z}\sGr^d$ consisting of chains of subspaces $E_\bull: \cdots \subset E_{-1}\subset E_0 \subset E_1 \subset \cdots$, with $E_d\subset V$ belonging to $\sGr^d$.  It is naturally a pro-ind-variety, and comes with projection morphisms $\pi_d \colon \sFl \to \sGr^d$.

The \define{shift automorphism} $\sh\colon V \to V$, defined by $\ee_i\mapsto \ee_{i-1}$, induces an automorphism of $\sFl$, by $\sh(E_\bull)_k = \sh(E_{k+1})$.  For a fixed positive $n$, the \define{affine flag variety} is the fixed locus of $\sh^n$:
\[
  \aFl_n = \{ E_\bull \in \sFl \,|\, \sh^n(E_\bull) = E_\bull \}.
\]
The \define{affine Grassmannian} is the image of $\aFl_n$ under the projection map:
\[
  \aGr_n^d = \pi_d(\aFl_n) = \{ E \in \sGr^d \,|\, \sh^n(E)\subset E \}.
\]

A torus $\TZ = \prod_{i\in\Z} \C^*$ acts on $V$ by scaling the coordinate $\ee_i$ by the character $y_i$.  This induces actions on $\sFl$ and $\sGr^d$.  We cyclically embed $\Tn = (\C^*)^n$ in $\TZ$, by specializing characters $y_i \mapsto y_{i\pmod n}$, using representatives $1,\ldots,n$ for residues mod $n$.  So $\Tn \subset \TZ$ is the fixed subgroup for the automorphism induced by $\sh^n$, and $\Tn$ therefore acts on $\aFl_n$ and $\aGr_n^d$.

The $\Tn$-fixed points of $\sFl$ (which are the same as the $\TZ$-fixed points) are indexed by the set $\Inj^0$ consisting of all injections $w\colon \Z \to \Z$ such that
\[
 \#\{ i\leq 0 \,|\, w(i)>0 \} = \#\{j>0\,|\, w(j)\leq 0\},
\]
and both these cardinalities are finite.\footnote{This implies $\#\{ i\leq d \,|\, w(i)>0 \} - \#\{j>d\,|\, w(j)\leq 0 \} = d$ for any integer $d$.}  The flag $E^w_\bull$ corresponding to $w\in\Inj^0$ consists of subspaces $E_k$ spanned by $\ee_{w(i)}$ for $i\leq k$, together with all $\ee_j$ for $j\leq 0$ not in the image of $w$.  The condition defining $\Inj^0$ guarantees $E_\bull^w$ lies in $\sFl$.  (See \cite[\S6]{infschub}.)

The $\Tn$-fixed points of $\aFl_n$ are indexed by the group of affine permutations.  This is the group $\tilde\Sgp_n$ consisting of bijections $w$ from $\Z$ to itself, such that $w(i+n)=w(i)+n$ for all $i\in\Z$, and such that $\sum_{i=1}^n w(i) = \binom{n}{2}$.  Among many other equivalent descriptions, this is the subset of $n$-shift-invariant elements in $\Inj^0$:
\[
  \tilde\Sgp_n = \{ w\in \Inj^0 \,|\, w(i+n)=w(i)+n \text{ for all }i\}.
\]

Similarly, $GL_n$ acts on $V$, extending the standard action on $V_{[1,n]}\isom \C^n$ by blocks, so $V = \cdots \oplus V_{[-n+1,0]}\oplus V_{[1,n]}\oplus V_{[n+1,2n]}\oplus \cdots$.  This induces actions on the Sato and affine flag varieties and Grassmannians.

Often we'll omit the superscript when focusing on the degree $d=0$ component, writing $\sGr = \sGr^0$ and $\aGr_n = \aGr_n^0$.

\subsection{Chern classes and cohomology}\label{ss.chern}

We write $c^{(d)}_k=c_k^{\Tn}(V_{\leq0}-\tS_d)$ in $H_{\Tn}^*\sGr^d$, and we use the same notation for the pullbacks to other varieties.  For $d=0$, or when the index is understood, we omit the superscript.  We have canonical isomorphisms
\[
  H_{\Tn}^*\sGr^d = \Lambda[y_1,\ldots,y_n] \quad \text{and} \quad H_{\Tn}^*\sFl=\Lambda[\ldots,x_{-1},x_0,x_1,\ldots;y_1,\ldots,y_n],
\]
where $\Lambda=\Z[c_1,c_2,\ldots]$ and $x_i=c_1^{\Tn}(\tS_i/\tS_{i-1})$ as before. (See \cite[\S3]{infschub}, but note that our sign convention on $x_i$ is opposite the one used there.)

For each fixed point $w\in\Inj^0$, there is a localization homomorphism $H_{\Tn}^*\sFl \to \Z[y_1,\ldots,y_n]$, given by
\[
  x_i \mapsto y_{w(i)} \qquad \text{and}\qquad c_k \mapsto [t^k]\Big(\mathop{\prod_{i\leq0, w(i)>0}}_{j>0,w(j)\leq0} \frac{1+y_{w(j)}t}{1+y_{w(i)}t} \Big).
\]
Here the operator $[t^k]$ extracts the coefficient of $t^k$, and we always understand $y_{a}$ as $y_{a \pmod n}$.  Since $w\in\Inj^0$, the RHS is a finite product.  The same formulas define localization homomorphisms for $\sGr$, $\aFl_n$, and $\aGr_n$.  For $\sGr^d$ and $\aGr^d_n$ with $d\neq 0$, we use
\[
  c^{(d)}_k \mapsto [t^k]\Big(\mathop{\prod_{i\leq d, w(i)>0}}_{j>d,w(j)\leq0} \frac{1+y_{w(j)}t}{1+y_{w(i)}t} \Big).
\]
We do not logically require these localization homomorphisms, but they are useful for checking that relations hold, and comparing them against other sources.

The shift morphism determines an automorphism $\gamma=\sh^*$ of $\Lambda[x,y]$, by
\[
  \gamma(x_i)=x_{i+1}, \quad \gamma(y_i)=y_{i+1}, \quad \text{and} \quad \gamma(C(t)) = C(t) \cdot \frac{1+y_1 t}{1+x_1 t},
\]
where $C(t) = \sum_{k\geq 0} c_k t^k$ is the generating series for $c$.

The inclusions $\aFl_n \hookrightarrow \sFl$ and $\aGr_n^d \hookrightarrow \sGr^d$ determine pullback homomorphisms on cohomology: we have maps
\[
 \Lambda[x;y] = H_{\Tn}^*\sFl \to H_{\Tn}^*\aFl_n \quad \text{and} \quad \Lambda[y] = H_{\Tn}^*\sGr^d \to H_{\Tn}^*\aGr^d_n.
\]
The main theorems assert that these homomorphisms are surjective, and specify the kernels.  One relation is immediately evident: since $\sh^n$ fixes $\aFl_n \subset \sFl$, we have $\gamma^n(c)=c$, so
\[
  \prod_{i=1}^n \frac{1+y_i t}{1+x_i t} = 1
\]
in $H_{\Tn}^*\aFl_n$.  As promised in the introduction, this shows that Corollary~B follows from the Theorem.

(An alternative argument uses the fact, not needed here, that the projection $\aFl_n \to \aGr_n$ is topologically identified with the trivial fiber bundle $\aGr_n \times \Fl(\C^n) \to \aGr_n$.)

\subsection{Coproduct}

There is a co-commutative coproduct structure on $\Lambda$, where the map $\Lambda \to \Lambda\otimes \Lambda$ is given by $c_k \mapsto c_k \otimes 1 + c_{k-1} \otimes c_1  +\cdots + 1 \otimes c_k$.  This extends $\Z[y]$-linearly to a coproduct on $\Lambda[y] = H_{\Tn}^*\sGr$.  As explained in \cite[\S8]{infschub}, this can be interpreted as an (equivariant) cohomology pullback via the direct sum morphism $\sGr \times \sGr \to \sGr$.

Likewise, there is a co-commutative coproduct structure on $H_{\Tn}^*\aGr_n$, coming from a homotopy equivalence with the based loop group, $\aGr_n \sim \Omega SU(n)$ (see \cite[\S8.6]{pressley-segal}).  The homotopy equivalence is equivariant with respect to the compact torus $(S^1)^n \subset \Tn$.  So the group structure on $\Omega SU(n)$ determines a coproduct on $H_{(S^1)^n}^*\Omega SU(n) = H_{\Tn}^*\aGr_n$.  (This coproduct can also be realized algebraically, but the construction is somewhat more involved than the direct sum map for $\sGr$---see, e.g., \cite{yun-zhu}.)

The coproducts on $H_\Tn^*\sGr$ and $H_\Tn^*\aGr_n$ are compatible, in the sense that the inclusion $\aGr_n \hookrightarrow \sGr$ induces a pullback homomorphism of co-algebras (and in fact, of Hopf algebras): the diagram
\[
\begin{tikzcd}
  H_\Tn^*\sGr \ar[r] \ar[d] & H_\Tn^*\sGr \otimes_{\Z[y]} H_\Tn^*\sGr \ar[d] \\
  H_\Tn^*\aGr_n \ar[r]  & H_\Tn^*\aGr_n \otimes_{\Z[y]} H_\Tn^*\aGr_n 
\end{tikzcd}
\]
commutes.

\section{Some algebra of symmetric functions}\label{s.Sym}

Before turning to the computation of equivariant cohomology rings, we review some basic facts about symmetric functions.  Much of what we need can be found in standard sources, e.g., \cite[Chapter I]{Mac1}.  We indicate proofs for facts not easily found there.

Recall $\Lambda=\Z[c_1,c_2,\ldots]$ and $\Lambda[y] = \Lambda[y_1,\ldots,y_n]$.

\subsection{Notation}

Let $\LambdaXi = \Z[\xi_1,\xi_2,\ldots]^{\Sgp_\infty}$ be the ring of symmetric functions in countably many variables $\xi_1,\xi_2,\ldots$, each of degree $2$.  This is the inverse limit of $\LambdaXi_r = \Z[\xi_1,\ldots,\xi_r]^{\Sgp_r}$ as $r\to \infty$ (in the category of graded rings).  It may be identified with the polynomial ring $\Z[h_1,h_2,\ldots]$, where $h_k=h_k(\xi)$ is the \define{complete homogeneous symmetric function} in $\xi$.

There is also a $\Z$-linear basis for $\LambdaXi$ consisting of the \define{monomial symmetric functions} $m_\lambda(\xi)$.  Given a partition $\lambda = (\lambda_1 \geq \lambda_2 \geq \cdots \geq \lambda_r\geq 0)$, the function $m_\lambda(\xi)$ is the symmetrization of the monomial $\xi_1^{\lambda_1} \xi_2^{\lambda_2} \cdots \xi_r^{\lambda_r}$---that is, the sum of all distinct permutations of this monomial.

The \define{power sum functions} $p_k(\xi) = \xi_1^k + \xi_2^k + \cdots$ also play an important role.  They generate $\LambdaXi$ as a $\Q$-algebra, but not as a $\Z$-algebra.  The function $p_k(\xi)$ is expressed in terms of the functions $h_k(\xi)$ via the {\it Newton relations}, which can be written as the determinant \eqref{e.Newton1}, substituting $h_k$ for $c_k$ in the matrix.

\subsection{An equality of ideals}

First we consider finitely many variables $\xi_1,\ldots,\xi_r$, and the symmetric polynomial ring $\LambdaXi_r \subset \Z[\xi_1,\ldots,\xi_r]$.

\begin{lemma}
Fix $n>0$, and consider the ideal $(\xi_1^n,\ldots,\xi_r^n)\subset \Z[\xi_1,\ldots,\xi_r]$.  As ideals in $\LambdaXi$, we have
\[
  (\xi_1^n,\ldots,\xi_r^n) \cap \LambdaXi_r = \big( m_\lambda(\xi) \big)_{\lambda_1\geq n} = \big( p_k(\xi) \big)_{k\geq n}.
\]
\end{lemma}

\begin{proof}
The first equality holds because monomials $\xi_1^{a_1}\cdots \xi_r^{a_r}$ with some $a_i\geq n$ form a $\Z$-linear basis for $(\xi_1^n,\ldots,\xi_r^n)\subset \Z[\xi_1,\ldots,\xi_r]$.  For the second equality, the inclusion ``$\supseteq$'' is evident, because $p_k = m_{(k)}$.  It remains to see that $m_\lambda$ lies in the ideal $(p_k)_{k\geq n}$ whenever $\lambda_1\geq n$, and this is proved by induction on the number of parts of $\lambda$.
\end{proof}

Taking the inverse limit over $r$ (in the category of graded rings), we obtain the following:
\begin{corollary}\label{c.sympn}
We have isomorphisms of graded rings
\begin{align*}
  \varprojlim_r \big( \Z[\xi_1,\ldots,\xi_r]/(\xi_1^n,\ldots,\xi_r^n) \big)^{\Sgp_r} &= \LambdaXi/\big( m_\lambda(\xi) \big)_{\lambda_1\geq n} \\
  &= \LambdaXi/\big( p_k(\xi) \big)_{k\geq n}.
\end{align*}
\end{corollary}

\subsection{Some identities in $\Lambda[y]$}\label{ss.ident}

We define elements $h_k\in\Lambda[y_1,\ldots,y_n]$ by
\begin{align}\label{e.h2c}
   h_k &= \sum_{i=0}^{k-1} \binom{k-1}{i} y_0^i\, c_{k-i},
\end{align}
writing $y_0=y_n$ to emphasize stability with respect to $n$.

Let $H(t)=\sum_{k\geq0} h_k t^k$ and $C(t) = \sum_{k\geq 0} c_k t^k$ be the generating series, with $h_0=c_0=1$.  Then \eqref{e.h2c} is equivalent to $H(t) = C\left( t/(1-y_0t) \right)$.  Both the $h$'s and the $c$'s are algebraically independent generators of $\Lambda[y]$ as a $\Z[y]$-algebra.

We define new elements $\tilde{p}_k \in \Z[h_1,h_2,\ldots]$ by an identity of generating series:
\begin{equation}
  \tilde{P}(t) := \sum_{k\geq 1} \tilde{p}_k t^{k-1} = \frac{d}{dt}\log H(t).
\end{equation}
Similarly, we have
\begin{equation}\label{e.plog}
  P(t) := \sum_{k\geq 1} p_k t^{k-1} = \frac{d}{dt}\log C(t).
\end{equation}
(These formulas are equivalent to the Newton relations \eqref{e.Newton1}; see, e.g., \cite[\S2]{Mac1}.)

\begin{remark}
There is an isomorphism $\Lambda[y]=\Z[c,y] \xrightarrow{\sim} \LambdaXi[y]$ determined by evaluating the generating series $C(t) = \sum c_k t^k$ as
\begin{equation}\label{e.c-eval}
  C(t) \mapsto \prod_{i\geq 1} \frac{1+y_0 t}{1-\xi_i t+ y_0t}.
\end{equation}
Under this identification, $H(t) = \prod_{i\geq 1} \frac{1}{1-\xi_i t}$, so $h_k$ maps to $h_k(\xi)$, and $\tilde{p}_k$ becomes the power sum function $p_k(\xi)= \xi_1^k + \xi_2^k + \cdots$.  
\end{remark}

Let $E(t) = \prod_{i=1}^n (1+ y_i t)$ be the generating series for the elementary symmetric polynomials in $y_1,\ldots,y_n$, and let $\tilde{E}(t) = \prod_{i=1}^n (1+(y_i-y_0)t)$ be the corresponding series in variables $y_i-y_0$.  So $\tilde{E}(t) = E(t/(1-y_0t))\cdot (1-y_0t)^n$.

Finally, let
\begin{align}
  p_k(c|y) &= p_k + p_{k-1} e_1(y) + \cdots + p_1 e_{k-1}(y) \label{e.pcydef}
\intertext{and}
 \tilde{p}_k(h|y) &= \tilde{p}_k + \tilde{p}_{k-1} e_1(y_1-y_0,\ldots,y_n-y_0) + \cdots \label{e.phydef} \\
   & \qquad + \tilde{p}_1 e_{k-1}(y_1-y_0,\ldots,y_n-y_0). \nonumber
\end{align}
Equivalently, the generating series for $p_k(c|y)$ and $\tilde{p}_k(h|y)$ are given by
\begin{equation*}
  \bm{P}(t)=P(t)\cdot E(t) \qquad \text{ and } \qquad \tilde{\bm{P}}(t) = \tilde{P}(t)\cdot\tilde{E}(t),
\end{equation*}
respectively.

We wish to compare the ideals generated by $p_k(c|y)$ and $\tilde{p}_k(h|y)$.

\begin{lemma}\label{l.pc2ph}
For $k\geq n$, we have
\begin{align}
 \tilde{p}_k(h|y) &= \sum_{i=0}^{k-1} \binom{k-n}{i} y_0^i p_{k-i}(c|y) \label{e.pc2ph}
\intertext{and}
 p_k(c|y) &= \sum_{i=0}^{k-1} \binom{k-n}{i} (-y_0)^i \tilde{p}_{k-i}(h|y). \label{e.ph2pc}
\end{align}
In particular, we have an equality
\[
  \Big( p_k(c|y) \Big)_{k\geq n} = \Big( \tilde{p}_k(h|y) \Big)_{k\geq n}
\]
of ideals in $\Lambda[y_1,\ldots,y_n]$.
\end{lemma}

\begin{proof}
The second statement follows from the first, the RHS of \eqref{e.pc2ph} involves only $p_i(c|y)$ for $i\geq n$, and likewise the RHS of \eqref{e.ph2pc} involves only $\tilde{p}_i(h|y)$ for $i\geq n$.

To prove \eqref{e.pc2ph}, we expand the definitions and compute:
\begin{align*}
  \tilde{\bm{P}}(t) &= \tilde{P}(t)\cdot \tilde{E}(t) \\
                  &= \left(\frac{d}{dt}\log H(t)\right)\cdot E\Big( t/(1-y_0t) \Big)\cdot (1-y_0t)^n \\
                  &= \left(\frac{d}{dt}\log C\Big( t/(1-y_0t) \Big)\right)\cdot E\Big( t/(1-y_0t) \Big)\cdot (1-y_0t)^n \\
                  &= \frac{1}{(1-y_0t)^2} P\Big( t/(1-y_0t) \Big) \cdot E\Big( t/(1-y_0t) \Big)\cdot (1-y_0t)^n \\
                  &= (1-y_0t)^{n-2}\bm{P}\Big( t/(1-y_0t) \Big).
\end{align*}
Expanding the RHS, we obtain
\[
  \sum_{m\geq 1} p_m(c|y) t^{m-1} (1-y_0 t)^{n-m-1} = \mathop{\sum_{m\geq 1}}_{i\geq 0} p_m(c|y) \binom{n-m-1}{i}(-y_0)^i t^{m-1+i}.
\]
Setting $k=m+i$, for $k\geq n$ the coefficient of $t^{k-1}$ is
\[
 \sum_{i=0}^{k-1} \binom{n-k+i-1}{i}(-y_0)^i p_{k-i}(c|y) = \sum_{i=0}^{k-1} \binom{k-n}{i}y_0^i p_{k-i}(c|y)
\]
as desired.  (The last equality uses the extended binomial coefficient identity $\binom{-m}{i} = (-1)^i\binom{m+i-1}{i}$.) The proof of \eqref{e.ph2pc} is analogous.
\end{proof}

\section{Proof of the main theorem}\label{s.bott}

Given any variety $X$ with basepoint $p_0$, Bott \cite{bott} considers a system of embeddings
\[
  X^{\times r} = X^{\times r} \times \{p_0\} \hookrightarrow X^{\times r+1}. 
\]
Assume $\Tn$ acts on $X$, fixing $p_0$, so these embeddings are $\Tn$-equivariant.  The symmetric group $\Sgp_r$ acts on these products by permuting factors, and therefore on their (equivariant) cohomology rings.  The inverse limit is written
\begin{equation}
  \Sgp H_{\Tn}^*X := \varprojlim_r \big( H_{\Tn}^*X^{\times r}\big)^{\Sgp_r}.
\end{equation}

We further assume $H_\Tn^*X$ is free over $\Z[y]=H_\Tn^*(\pt)$, and has no odd cohomology.    Then $H_\Tn^*X^{\times r} = H_\Tn^*X \otimes_{\Z[y]} \cdots \otimes_{\Z[y]} H_\Tn^*X$ ($r$ factors).  In this case, given any $\Tn$-equivariant morphism $f\colon X \to \aGr_n$, there is a pullback homomorphism
\[
 H_{\Tn}^*\aGr_n \to H_{\Tn}^*X \otimes_{\Z[y]} \cdots \otimes_{\Z[y]} H_\Tn^*X,
\]
obtained by factoring through the $r$-fold coproduct on $H_{\Tn}^*\aGr_n$.  Since the coproduct is commutative, the image lies in the $\Sgp_r$-invariant part of the tensor product.  Taking the limit over $r$ produces a homomorphism
\[
  f^* \colon H_{\Tn}^*\aGr_n \to \Sgp H_{\Tn}^*X.
\]

For $X$, we take projective space $\P(V_{[0,n-1]}) \isom \P^{n-1}$, with basepoint $p_0$ corresponding to the line $\C\cdot \ee_0 \subset V_{[0,n-1]}$, which is scaled by the character $y_0=y_n$.  (Recall that we treat indices of $y_i$ modulo $n$.)

Let $H = \P(V_{[1,n-1]}) \subset \P(V_{[0,n-1]}) = \P^{n-1}$ be the hyperplane defined by $e^*_0=0$, and let $\xi = [H]$ be its class in $H_{\Tn}^2\P^{n-1}$.  So $\xi =c_1^\Tn(\O(1)) + y_0$, where $\O(1)$ is the dual of the tautological bundle on $\P^{n-1}$.  The equivariant cohomology ring of $\P^{n-1}$ has a well-known presentation, which in our notation takes the form
\begin{equation}\label{e.HTPn}
  H_{\Tn}^*\P^{n-1} = \Z[y][\xi]/\big( \xi(\xi+y_1-y_0)\cdots(\xi+y_{n-1}-y_0) \big).
\end{equation}
Written slightly differently, the defining relation is
\begin{align}
  \xi^n + \xi^{n-1}\, e_1(y_1-y_0,\ldots,y_n-y_0) + \cdots \label{e.HTPn2} \\ \qquad + \xi\, e_{n-1}(y_1-y_0,\ldots,y_n-y_0) = 0, \nonumber
\end{align}
which one should compare with \eqref{e.phydef}.  
Similarly, let $H_i \subset (\P^{n-1})^{\times r}$ be the hyperplane defined by $e^*_0=0$ on the $i^\mathrm{th}$ factor, and let $\xi_i = [H_i]$ be its class in $H_\Tn^*(\P^{n-1})^{\times r}$, which has a presentation with one relation of the form \eqref{e.HTPn2} for each $\xi_i$.  Taking symmetric invariants leads to the following calculation:

\begin{lemma}\label{l.SHTP}
The ring $\Sgp H_{\Tn}^*\P^{n-1}$ is a free $\Z[y]$-algebra.  Letting $\tilde{p}_k(\xi|y)$ be the polynomials defined by \eqref{e.phydef}, where $\tilde{p}_k=p_k(\xi)=\xi_1^k+\xi_2^k+\cdots$, it has the presentation
\[
  \Sgp H_{\Tn}^*\P^{n-1} = \LambdaXi[y]/\big( \tilde{p}_k(\xi|y) \big)_{k\geq n}.
\]
\end{lemma}

\begin{proof}
The homomorphism $\LambdaXi[y] \to \Sgp H_{\Tn}^*\P^{n-1}$ is the limit of homomorphisms $\Z[\xi_1,\ldots,\xi_r]^{\Sgp_r} \to \big(H_{\Tn}(\P^{n-1})^{\times r} \big)^{\Sgp_r}$ defined by $\xi_i\mapsto [H_i]$.  
The relations $\tilde{p}_k(\xi|y)=0$ hold in $\Sgp H_{\Tn}^*\P^{n-1}$, because they symmetrize relations of the form \eqref{e.HTPn2}, so there is a well-defined homomorphism modulo the ideal $\big( \tilde{p}_k(\xi|y) \big)_{k\geq n}$.  Modulo the $y$-variables, this reduces to the isomorphism described in Corollary~\ref{c.sympn}.  The assertion follows by graded Nakayama.
\end{proof}

One embeds $\P(V_{[0,n-1]})$ in $\sGr$ by sending $L \subset V_{[0,n-1]}$ to $V_{<0}\oplus L \subset V$, and this embedding factors through $\aGr_n$, all $\Tn$-equivariantly.  So we have homomorphisms
\begin{equation}\label{e.homs}
\Lambda[y] = H_{\Tn}^*\sGr \to  H_{\Tn}^*\aGr_n \xrightarrow{f^*} \Sgp H_{\Tn}^*\P^{n-1}.
\end{equation}
The map $\Lambda[y]=H_\Tn^*\sGr \to H_\Tn^*\P^{n-1}$ sends the generating series
\[
  C(t)=c^\Tn(V_{\leq0}-\tS)\quad \text{ to } \quad c^\Tn(\C\cdot \ee_0 - \O(-1)) = \frac{1+y_0 t}{1-\xi t + y_0t}.
\]
The map $\Lambda[y] \to \Sgp H_{\Tn}^*\P^{n-1}$ is determined by the evaluation \eqref{e.c-eval}.

\begin{proposition}\label{p.bott}
The homomorphism $f^*\colon H_{\Tn}^*\aGr_n \to \Sgp H_{\Tn}^*\P^{n-1}$ is an isomorphism of $\Z[y]$-algebras.  In particular, we have
\[
  H_{\Tn}^*\aGr_n = \Lambda[y]/\big( \tilde{p}_k(h|y) \big)_{k\geq n}.
\]
\end{proposition}

\begin{proof}
The affine Grassmannian has a $\Tn$-invariant Schubert cell decomposition, with finitely many cells in each dimension, so $H_{\Tn}^*\aGr_n$ is a free $\Z[y]$-module.  It follows that non-equivariant cohomology is recovered by setting $y$-variables to $0$: we have an isomorphism $(H_{\Tn}^*\aGr_n)/(y) \isom H^*\aGr_n$, and likewise $(\Sgp H_{\Tn}^*\P^{n-1})/(y) \isom \Sgp H^*\P^{n-1}$.  The induced map $H^*\aGr_n \to \Sgp H^*\P^{n-1}$ was shown to be an isomorphism by Bott \cite[Proposition~8.1]{bott}.  So the first statement of the proposition follows by another application of graded Nakayama.  The second statement is a combination of the first with the presentation of $\Sgp H_{\Tn}^*\P^{n-1}$ from Lemma~\ref{l.SHTP}.
\end{proof}

The $d=0$ case of the main theorem follows from Proposition~\ref{p.bott} together with the equality of ideals $\big( \tilde{p}_k(h|y) \big)_{k\geq n} = \big( p_k(c|y) \big)_{k\geq n}$ established in Lemma~\ref{l.pc2ph}.

For the general $d$ case, we use the shift morphism $\sh^d$, which defines isomorphisms
\[
\begin{tikzcd}
  \sGr^d \ar[r,"\sim"] & \sGr \\
  \aGr_n^d \ar[r,"\sim"] \ar[u,hook] & \aGr_n \ar[u,hook].
\end{tikzcd}
\]
These are equivariant with respect to the corresponding automorphism of $\Tn$ which cyclically permutes coordinates.  The action on cohomology rings is given by the homomorphism $\gamma^d$, as described in \S\ref{ss.chern}.  The presentation of $H_{\Tn}^*\aGr_n$ is mapped to
\[
  H_{\Tn}^*\aGr_n^d = \Lambda[y]/\big( \gamma^d p_k(c|y) \big)_{k\geq n},
\]
where now $\Lambda=\Z[c_1^{(d)},c_2^{(d)},\ldots]$, and the variables map by $c_k^{(d)} = c^{\Tn}_k(V_{\leq0}-\tS_d)$.  It remains to express $\gamma^d p_k(c|y)$ in terms of the polynomials $p_k(c^{(d)}|y)$.

Since $(\sh^d)^*c_k^{\Tn}(V_{\leq 0} - \tS_0) = c_k^{\Tn}(V_{\leq d} - \tS_d)$, we have
\[
\gamma^d( C(t) ) = C^{(d)}(t) \cdot (1+y_1 t) \cdots (1+ y_d t),
\]
where $C^{(d)}(t) = \sum_{k\geq 0} c_k^{(d)} t^k$ is the generating series.  So, using notation from \S\ref{s.Sym}, we have
\begin{align*}
 \gamma^d\bm{P}(t) &= \big(\gamma^d P(t) \big)\cdot \big(\gamma^d E(t) \big) \\
                   &= \left( \frac{d}{dt}\log \gamma^d C(t) \right) \cdot E(t) \\
                   &= \frac{d}{dt}\log \left( C^{(d)}(t)\prod_{i=1}^d (1+y_i t) \right) \cdot E(t) \\
                   &= P^{(d)}(t) \cdot E(t) + \sum_{i=1}^d y_i (1+y_1 t) \cdots \hat{(1+y_i t)} \cdots (1+y_n t),
\end{align*}
where $P^{(d)}(t) = \sum_{k\geq 1} p_k(c^{(d)}|y) t^{k-1}$.  Extracting the coefficients of $t^{k-1}$, we find
\begin{align*}
  \gamma^d p_n(c|y) &= p_n(c^{(d)}|y) + d\cdot e_n( y_1,\ldots, y_n )
  \intertext{and}
   \gamma^d p_k(c|y) &= p_k(c^{(d)}|y)
\end{align*}
for $k>n$, as claimed.

\begin{remark}
Consider the $\Z[y]$-algebra automorphism of $\Lambda[y]$ defined by sending $p_k(c)$ to $p_k(c) - (-1)^k p_k(y)$, where $p_k(y) = y_1^k+\cdots+y_n^k$.  Using \cite[(2.11')]{Mac1}, this sends
\[
  p_k(c|y) \mapsto p_k(c|y) + k\,e_k(y).
\]
So we have an isomorphism of $\Z[y]$-algebras $\Lambda[y]/I_n^d \xrightarrow{\sim} \Lambda[y]/I_n^{d+n}$.
\end{remark}

\section{Double monomial symmetric functions}\label{s.monomial}

The monomial symmetric functions $m_\lambda(\xi)$, with $\lambda_1<n$, form a basis for $\Sgp H_{\Tn}^*\P^{n-1}$ over $\Z[y]$---so they also form a basis for $H_{\Tn}^*\aGr_n$.  (This follows from the arguments above, and it is also easy to see directly from the fact that $1,\xi,\ldots,\xi^{n-1}$ forms a basis for $H_{\Tn}^*\P^{n-1}$ over $\Z[y]$.)  It is useful to work with a deformation of this basis of $\Lambda[y]$, which extends a basis for the defining ideal of $H_{\Tn}^*\aGr_n$.

For the general definition, we use variables $a_1,a_2,\ldots$ in degree $2$.  Given a sequence $\alpha=(\alpha_1,\ldots,\alpha_r)$ of positive integers, let $n_i(\alpha)$ be the number of occurrences of $i$ in $\alpha$, and set $n(\alpha) := n_1(\alpha)! n_2(\alpha)!\cdots$.  (So $n(\alpha)$ is the number of permutations fixing $\alpha$.)  For a partition $\lambda$ with $r$ parts, so $\lambda=(\lambda_1\geq \cdots \geq \lambda_r>0)$, we write $\alpha\subset \lambda$ to mean $\alpha_i\leq \lambda_i$ for all $i$.  Let
\[
  e_{\lambda-\alpha}(a) = e_{\lambda_1-\alpha_1}(a_1,\ldots,a_{\lambda_1-1})\cdots e_{\lambda_r-\alpha_r}(a_1,\ldots,a_{\lambda_r-1}),
\]
where $e_k$ is the elementary symmetric polynomial.

\begin{definition}
The \define{double monomial symmetric function} is
\[
  m_\lambda(\xi|a) = \sum_{(1^r) \subset \alpha \subset \lambda} \frac{n(\alpha)}{n(\lambda)} e_{\lambda-\alpha}(a)\, m_\alpha(\xi),
\]
an element of $\LambdaXi[a_1,a_2,\ldots]$.
\end{definition}

For a given $\alpha$, the coefficient $n(\alpha)/n(\lambda)$ need not be an integer, but in the sum over all $\alpha$, the coefficients are integers.  In fact, $m_\lambda(\xi|a)$ is the symmetrization of the ``monomial''
\begin{align}\label{e.xi-monomial}
  (\xi|a)^\lambda &= \prod_{i=1}^r \xi_i(\xi_i+a_1)\cdots(\xi_i+a_{\lambda_i-1}) \\
                  &= \sum_{(1^r)\subset \alpha \subset \lambda} e_{\lambda-\alpha}(a)\, \xi^\alpha, \nonumber
\end{align}
i.e., it is the sum of $\sigma\big((\xi|a)^{\lambda}\big)$ over all distinct permutations $\sigma$ of $\lambda$, where $\sigma$ acts in the usual way by permuting the $\xi$ variables.

For instance, the functions corresponding to $\lambda$ with a single row are
\[
  m_{k}(\xi|a) = m_{k}(\xi) + e_1(a_1,\ldots,a_{k-1})\,m_{k-1}(\xi) + \cdots + e_{k-1}(a_1,\ldots,a_{k-1})\, m_1(\xi).
\]
Other examples are:
\begin{align*}
 m_{21}(\xi|a) &= m_{21}(\xi) + 2a_1\,m_{11}(\xi), \\
 m_{22}(\xi|a) &= m_{22}(\xi) + a_1\, m_{21}(\xi) + a_1^2\, m_{11}(\xi), \\
 m_{31}(\xi|a) &= m_{31}(\xi) + (a_1+a_2)\, m_{21}(\xi) + 2 a_1 a_2\, m_{11}(\xi), \\
 m_{32}(\xi|a) &= m_{32}(\xi) + 2(a_1+a_2)\, m_{22}(\xi)  + a_1\, m_{31}(\xi)  \\
  & \qquad + a_1 (a_1+2a_2)\, m_{21}+ 2 a_1^2 a_2\, m_{11}(\xi).
\end{align*}

From now on, we evaluate the $a$ variables as $a_i=y_i-y_0$, with the indices taken mod $n$ as usual.  In the single-row case, this recovers the double power sum function defined by \eqref{e.phydef} in \S\ref{s.Sym} above: $m_k(\xi|a)=\tilde{p}_k(\xi|y)$.

We use the isomorphism $\LambdaXi[y] \isom \Lambda[y]$ from \eqref{e.c-eval} to identify the functions $m_\lambda(\xi|y)$ in $\LambdaXi[y]$ with elements $m_\lambda(c|y)$ in $\Lambda[y]$, also called double monomial functions.  

\begin{proposition}
The double monomial functions $m_\lambda(c|y)$ form a $\Z[y]$-linear basis for $\Lambda[y]$.  The $m_\lambda(c|y)$ with $\lambda_1\geq n$ form a $\Z[y]$-linear basis for the ideal $I_n \subset \Lambda[y]$, the kernel of the surjective homomorphism $\Lambda[y]=H_{\Tn}^*\sGr \to H_{\Tn}^*\aGr_n$.
\end{proposition}

In particular, every class in $H_{\Tn}^*\aGr_n$ has a canonical lift to a polynomial in $\Lambda[y]$, by taking an expansion in the monomial basis as a normal form, using only those $m_\lambda(c|y)$ with $\lambda_1<n$.

\begin{proof}
The first statement is proved by setting $y=0$, since the monomial functions $m_\lambda$ form a basis for $\Lambda$.  For the second statement, it suffices to check that each $m_\lambda(c|y)$ lies in the ideal.  This follows from the characterization of $m_\lambda(\xi|a)$ as the symmetrization of the monomial $(\xi|a)^\lambda$ defined in \eqref{e.xi-monomial}.  Indeed, after setting $\xi_i = [H_i]$ and $a_i = y_i-y_0$, as in \S\ref{s.bott}, each $(\xi|a)^\lambda$ with $\lambda_1\geq n$ lies in the ideal defining $H_{\Tn}^*(\P^{n-1})^{\times r}$, so the symmetrization lies in the defining ideal of $\Sgp H_{\Tn}^*\P^{n-1}$.
\end{proof}

\begin{remark}
Up to sign and reindexing variables, the single-row functions $m_{k}(\xi|a)$ are essentially the same as the functions $\tilde{m}_k(x|\!|a)$ in \cite[\S4.5]{lam-shimozono}.  (To make the identification, use an isomorphism of our $\LambdaXi[a]$ with their $\Lambda(x|\!|a)$ which sends $m_k(\xi) \mapsto m_k[x-a_{>0}]$ and $a_i \mapsto -a_{1-i}$.  Then the image of our $m_k(\xi|a)$ is the result of setting $a_1=0$ in $\tilde{m}_k(x|\!|a)$.)  In general, however, the double monomial functions defined here differ from those of \cite{lam-shimozono}, which are more analogous to power-sum functions.  For instance, the latter are a basis only over $\Q[a]$.

The $m_\lambda(\xi|a)$ are closer to the double monomial functions $m_\lambda(x|\!|a)$ introduced by Molev \cite[\S5]{molev}, which are defined non-explicitly via Hopf algebra duality, but do form a basis over $\Z[a]$.  They are not quite identical, as can be seen from the table in \cite[\S8.1]{lam-shimozono}, but in small examples the image of our $m_\lambda(\xi|a)$ under the substitution $a_i\mapsto -a_{1-i}$ agrees with the result of setting $a_1=0$ in Molev's function $m_\lambda(x|\!|a)$.  
It would be interesting to know if this pattern persists.
\end{remark}

\section{Moduli of vector bundles}\label{s.VBs}

The affine Grassmannian $\aGr_n^d$ is homotopy-equivalent to the moduli stack parameterizing rank-$n$, degree $d$ vector bundles on $\P^1$ together with a trivialization at $\infty$.  Forgetting the trivialization identifies the moduli stack of vector bundles on $\P^1$ with the quotient stack $[GL_n\backslash\aGr_n^d]$.  (See, e.g., \cite{larson} for constructions of the moduli stacks, as well as further references, and \cite[\S4]{zhu} for a careful exposition of the relation between moduli of bundles and affine Grassmannians.)

Larson gave an algebraic description of the Chow ring of the moduli stack $\mathcal{B}^\dagger_{n,d}$ of rank $n$, degree $d$ vector bundles on $\P^1$, as a certain subring of a polynomial ring \cite{larson}.  In our context, the Chow and singular cohomology rings are isomorphic, and it follows from the above considerations that this ring must be isomorphic to the equivariant cohomology ring $H_{GL_n}^*\aGr_n^d$.  Here we will show that Larson's description is equivalent to the presentation given above in Corollary~A, using some basic identities of symmetric functions.

Consider the polynomial ring $\Q[e_1,\ldots,e_n,\bq_1,\ldots,\bq_{n-1}]$, with $e_i$ and $\bq_i$ in degree $2i$.  Larson shows that $H^*\mathcal{B}^\dagger_{n,d} = H_{GL_n}^*\aGr_n^d$ is isomorphic to the subring generated over $\Z[e_1,\ldots,e_n]$ by the coefficients of a series $\bar{C}(t) = \sum_{k\geq 0} \bar{c}_k\, t^k$, defined by
\begin{equation}\label{e.c-bar}
  \exp\left( \int \frac{-d(e_1+e_2 t+ \cdots + e_n t^{n-1}) + (\bq_1+\bq_2\, t + \cdots + \bq_{n-1}\, t^{n-2})}{1 + e_1 t + \cdots + e_n t^{n}}\,dt \right).
\end{equation}
(To compare with Larson's notation, our $\bar{c}_i$ is her $e_i$, our $e_i$ is her $a_i$, and our $\bq_i$ is her $-a'_{i+1}$.)

\begin{proposition}
The ideal $J_n^d$ is the kernel of the $\Z[e_1,\ldots,e_n]$-algebra homomorphism $\Lambda[e_1,\ldots,e_n] \to \Q[e_1,\ldots,e_n,\bq_1,\ldots,\bq_{n-1}]$ which sends $c_k$ to $\bar{c}_k$.  In particular, the $\Z[e_1,\ldots,e_n]$-subalgebra of $\Q[e_1,\ldots,e_n,\bq_1,\ldots,\bq_{n-1}]$ generated by the $\bar{c}_k$ is isomorphic to $\Lambda[e_1,\ldots,e_n]/J_n^d \isom H_{GL_n}^*\aGr_n^d$.
\end{proposition}

\begin{proof}
Consider a generating series
\[
  \bQ(t) = \sum_{k>0} \bq_k t^{k-1},
\]
along with 
\begin{equation}\label{e.c-p}
 C(t) = \exp\left( \int \frac{-d(e_1+e_2 t+ \cdots + e_n t^{n-1}) + \bQ(t)}{E(t)}\,dt \right),
\end{equation}
where $E(t) = \sum_{k=0}^n e_k t^k$ as usual.  The coefficients $c_k$ are algebraically independent, so this formula defines an embedding $\Lambda[e_1,\ldots,e_n] \hookrightarrow \Q[e_1,\ldots,e_n,\bq_1,\bq_2,\ldots]$.  The elements $\bar{c}_k$ defined by \eqref{e.c-bar} are the images of $c_k$ under the projection
\[
   \Q[e_1,\ldots,e_n,\bq_1,\bq_2,\ldots] \to \Q[e_1,\ldots,e_n,\bq_1,\ldots,\bq_{n-1}]
\]
which sets $\bq_k$ to $0$ for $k\geq n$.  So it suffices to identify these $\bq_k$ with the generators of $J_n^d$.

Rewriting the expression \eqref{e.c-p}, we find
\[
  t\,\bQ(t) = t\, P(t) \, E(t) + d\big( E(t)-1\big ),
\]
where the series $P(t) = \frac{d}{dt}\log C(t)$ is determined by the Newton relations, in the form given in \eqref{e.plog}.  Extracting the coefficient of $t^k$, we see $\bq_k = p_k(c|e)+d\, e_k$ for all $k\geq 1$.  In particular, $\bq_n = p_n(c|e)+ d\, e_n$, and $\bq_k = p_k(c|e)$ for $k>n$.
\end{proof}



\end{document}